\documentclass{article}

\usepackage{amsmath}
\usepackage{amssymb}
\usepackage{amsthm}
\usepackage{graphics}
\usepackage{a4wide}

\newtheorem{theorem}{Theorem}[section]
\newtheorem{definition}[theorem]{Definition}
\newtheorem{proposition}[theorem]{Proposition}
\newtheorem{corollary}[theorem]{Corollary}
\newtheorem{lemma}[theorem]{Lemma}
\newtheorem{fact}[theorem]{Remark}

\newtheorem{notation}[theorem]{Notation}

\title{Quantitative results on the Ishikawa iteration of Lipschitz pseudo-contractions}
\author{Lauren\c{t}iu Leu\c{s}tean$^{1,2}$, Vlad Radu$^{1}$ and Andrei Sipo\c s$^{1,2}$\\[0.2cm]
\footnotesize ${}^1$ Faculty of Mathematics and Computer Science, University of Bucharest,\\
\footnotesize Academiei 14,  P.O. Box 010014, Bucharest, Romania\\[0.1cm]
\footnotesize ${}^2$ Simion Stoilow Institute of Mathematics of the Romanian Academy,\\
\footnotesize P. O. Box 1-764, 014700 Bucharest, Romania\\[0.1cm]
\footnotesize E-mails:  laurentiu.leustean@unibuc.ro, vlad.radu2013@yahoo.com,  Andrei.Sipos@imar.ro
}

\date{}

\begin{document}

\maketitle

\begin{abstract}
\noindent  We compute uniform rates of metastability for the Ishikawa iteration of a Lipschitz pseudo-contractive
self-mapping of a compact convex subset of a Hilbert space. This extraction is an instance of the proof mining
program that aims to apply tools from mathematical logic in order to extract the hidden quantitative content
of mathematical proofs. We prove our main result by applying methods developed by Kohlenbach, the first  author and 
Nicolae for obtaining quantitative versions of strong convergence results for generalized Fej\'er monotone sequences in compact subsets of metric spaces.\\



\end{abstract}

\section{Introduction}

Let $H$ be a real Hilbert space, $C\subseteq H$ a nonempty convex subset and $T:C\to C$ be a mapping.  

We say that $T$ is a \emph{pseudo-contraction} if for all $x,y\in C$,
\begin{equation} 
 \|Tx - Ty\|^2 \leq \|x-y\|^2 + \|(x-Tx)-(y-Ty)\|^2. \label{def-pseudo-contraction}
\end{equation}
This class of nonlinear mappings was introduced in the 1960s by Browder and Petryshyn \cite{BroPet67}. 
Its significance lies in the following fact: an operator $T$ is a pseudo-contraction if and only if its complement $U:=Id - T$ is monotone, i.e. for all $x,y \in C$ we have that
$$\langle Ux-Uy, x -y \rangle \geq 0.$$

Monotone operators arise naturally in the study of partial differential equations: often such an 
equation can be written in the form $U(x) = 0$ (or $0 \in U(x)$ when considering multi-valued operators). 
Finding a zero of $U$ is equivalent to finding a fixed point of its complement $T:= Id - U$, hence the problem 
of finding fixed points of nonlinear operators is tightly linked to that of finding solutions to nonlinear equations.\\[1mm]
It is well-known that the classical method of Picard iterations, used to find the unique 
fixed point of a contraction, fails in the case of nonexpansive mappings, i.e. maps that only 
satisfy $\|Tx-Ty\|\leq\|x-y\|$ for all $x,y\in C$. Nevertheless, by considering an iteration of the form
$$x_0:=x, \quad x_{n+1} := \alpha_nTx_n + (1-\alpha_n)x_n,$$
where $(\alpha_n)_{n\in{\mathbb N}}$ is a sequence in $[0,1]$ satisfying some mild conditions, one 
obtains a sequence that converges (in some cases only weakly) to a fixed point of $T$. 
Such a scheme is called the {\em Mann iteration}. Efforts to extend this scheme to more general 
maps like pseudo-contractions were not successful. 
Later, Chidume and Mutangadura \cite{ChiMut01} would exhibit an example of a Lipschitzian 
pseudo-contractive map with a unique fixed point for which no Mann sequence converges. 

We recall that  $T$ is said to be {\em $L$-Lipschitzian} (for an $L>0$) if for all $x,y \in C$ we have that 
$\|Tx-Ty\|\leq L\|x-y\|$. Examples of Lipschitzian pseudo-contractions are strict 
pseudo-contractions (defined also in \cite{BroPet67}), hence, in particular, nonexpansive mappings.\\[1mm] 
Meanwhile, some alternate algorithms were proposed, the first of which being the one of 
Ishikawa \cite{Ish74}, who deployed it successfully in the case of Lipschitzian pseudo-contractions 
acting on a compact convex subset of a Hilbert space. It is defined as follows.

If $(\alpha_n)_{n\in{\mathbb N}},(\beta_n)_{n\in{\mathbb N}}$ are sequences in $[0,1]$, then the {\em Ishikawa iteration} starting with an 
$x\in C$ using the two sequences as weights is defined by:
\begin{equation}
x_0:=x, \quad x_{n+1}:= \alpha_nT(\beta_nTx_n+(1-\beta_n)x_n) + (1-\alpha_n)x_n. \label{def-Ishikawa}
\end{equation}
We recognize the Mann iteration in the special case where $\beta_n:=0$ for all $n\in{\mathbb N}$.

We introduce the following conditions that sequences $(\alpha_n)$, $(\beta_n)$ in $[0,1]$ may satisfy: 
\begin{eqnarray*}
\text{(A1)} & \lim\limits_{n\to\infty}\beta_n=0;\\
\text{(A2)} &  \sum\limits_{n=0}^{\infty}\alpha_n\beta_n = \infty;\\
\text{(A3)} & \alpha_n\leq\beta_n,\ \text{for all }n \in {\mathbb N}.
\end{eqnarray*}
As pointed out in \cite{Ish74}, an example of a pair of sequences satisfying all three conditions 
is $\alpha_n=\beta_n=\frac1{\sqrt{n+1}}$. 

We can now state the exact form of Ishikawa's 1974 strong convergence result for the above iteration.

\begin{theorem}\label{main-thm-Ishikawa}
Let $H$ be a Hilbert space, $C\subseteq H$ a nonempty convex compact subset, $T:C\to C$ a Lipschitzian 
pseudo-contraction and $(\alpha_n)$, $(\beta_n)$ sequences in $[0,1]$ that satisfy (A1)-(A3). 
Then, for all $x\in C$, the Ishikawa iteration starting with $x$, using $(\alpha_n)$ and $(\beta_n)$ as weights,
converges strongly to a fixed point of $T$.
\end{theorem}

Note that Ishikawa, in the above result, does not assume {\it a priori} the existence of fixed points for $T$ -- 
this follows because of the compactness assumption of $C$, by an application of the theorem of Schauder. 
In order to obtain this strong convergence result in its quantitative form, as it is done in the 
last section of this paper, one must preserve this compactness assumption (in the quantitative form
of a modulus of total boundedness, as we shall see). However, compactness is not 
needed to obtain the preliminary result of the modulus of liminf - only the fixed point assumption (see 
Section \ref{modulus-liminf}). \\[1mm]
As suggested above, our goal in this paper is to obtain a quantitative version of 
Theorem~\ref{main-thm-Ishikawa} using methods of proof mining developed in \cite{KohLeuNic15}. 
The research program of proof mining in mathematical logic -- first suggested by G. Kreisel in the 1950s as `unwinding of proofs'
and given maturity by U. Kohlenbach in the 1990s and afterwards -- has developed into a field of 
study that aims to analyze, using tools from mathematical logic, the proofs of existing mathematical theorems 
in order to obtain their hidden quantitative content. A number of `logical metatheorems' guarantee 
that in situations that cover a significant portion of target theorems this sort of proof 
analysis can actually be done and the bounds obtained are highly uniform. A comprehensive 
reference for proof mining and its applications up to 2008 is \cite{Koh08-book}, while a 
recent survey is \cite{Koh16}. We point out also that the Ishikawa 
iteration was already approached with proof mining methods in \cite{Leu10,Leu14} for nonexpansive 
mappings in uniformly convex geodesic spaces.\\[1mm]
In our case, i.e. when analysing Ishikawa's above result, whose conclusion states that a sequence converges,
a quantitative version would be a rate of convergence that computes the corresponding $N_\varepsilon$ 
given the $\varepsilon$ and perhaps some additional parameters.
However, the high logical complexity of the definition of convergence makes it intractable for proofs 
that involve some notion of excluded middle, as it is the case here. Therefore, an equivalent 
formulation (identifiable in logic as its Herbrand normal form) introduced in this case by 
Tao \cite{Tao07,Tao08} under the name of {\em metastability}, is used in its stead. The following 
sentence expresses the metastability of a given sequence $(x_n)$ in a normed space:
$$\forall k \in {\mathbb N} \,\forall g: {\mathbb N} \to {\mathbb N} \,\exists N \in {\mathbb N} \,\forall i,j \in [N, N+g(N)]\,\,\, 
\left(\|x_i - x_j\| \leq \frac1{k+1}\right).$$
One can immediately glimpse the reduced complexity of this statement: no unbounded 
universal quantifier occurs after the existential one (as it clearly does in the usual 
formulations of convergence or Cauchyness). It is a simple exercise, however, to check 
that the sentence is equivalent to the assertion that $(x_n)$ is Cauchy -- and one should 
note that an appeal to {\it reductio ad absurdum} is inevitable in the process. The main 
result of this paper, Theorem \ref{main-thm-meta}, exhibits an effective rate of 
metastability -- that is, a bound $\Omega(k,g)$ on the $N$ in the above formulation -- for 
the Ishikawa iteration. \\[1mm]
The next section enumerates and proves some basic properties of the relevant mappings and sequences. 
Section \ref{modulus-liminf} contains a quantitative version of the first step of Ishikawa's proof, 
namely the modulus of liminf for $(\|x_n-Tx_n\|)$, which also serves to obtain the approximate 
fixed point bound, one of the necessary ingredients in the final analysis. The other ones are the moduli 
of uniform closedness and uniform Fej\'er monotonicity, introduced in \cite{KohLeuNic15}. 
The corresponding definitions can be found in Section \ref{Fejer-closed}, along with the 
concrete values of them for the case at hand. All these are put together in the last section, 
where the main result is stated and proved.

\mbox{}

\begin{tabular}{ll}
{\bf Notation}: & ${\mathbb N}=\{0,1,2,\ldots\}$ and $[m,n]=\{m,m+1,\ldots,n\}$ for any $m, n\in{\mathbb N}$ with $m\leq n$.\\[1mm]
& $Fix(T)$ is the set of fixed points of $T$.
\end{tabular}

\section{Some useful lemmas}

Let $H$  be a Hilbert space, $C\subseteq H$ a nonempty convex subset
and $T:C\to C$  be a mapping. Furthermore, $(\alpha_n)$ and $(\beta_n)$ are sequences of reals in $[0,1]$ and $(x_n)$ is the Ishikawa iteration 
starting with $x\in C$, defined by \eqref{def-Ishikawa}, using $(\alpha_n)$ and $(\beta_n)$ as weights.

In order for the computations to be less cumbersome, we shall also set for all $n \in {\mathbb N}$,
\[y_n := \beta_nTx_n+(1-\beta_n)x_n,\]
so that we have, again for all $n \in {\mathbb N}$,
\[x_{n+1}=(1-\alpha_n)x_n+ \alpha_nTy_n.\]

\begin{fact}\label{imm}
It is clear that $x_n - x_{n+1} = \alpha_n(x_n-Ty_n)$, so $\|x_n - x_{n+1}\| \leq \|x_n - Ty_n\|$, and that $x_n - y_n = \beta_n(x_n - Tx_n)$, so $\|x_n - y_n\| \leq \|x_n - Tx_n\|$.
\end{fact}

\begin{lemma}\label{lem-1plusL}
Assume that $T$ is $L$-Lipschitzian. Then $\|x_n - x_{n+1}\| \leq (1+L)\|x_n - Tx_n\|$.
\end{lemma}
\begin{proof}
Using Remark~\ref{imm}, we have that:
\begin{eqnarray*}
\|x_n - x_{n+1}\| &\leq &  \|x_n - Ty_n\| \leq \|x_n - Tx_n\| + \|Tx_n - Ty_n\| 
\leq \|x_n - Tx_n\| + L\|x_n - y_n\| \\
 &\leq &  (1+L)\|x_n-Tx_n\|.
\end{eqnarray*}
\end{proof}

We recall the following well-known and useful equalities that hold in Hilbert spaces.

\begin{lemma}\label{Hilbert-id}
For any $x,y\in H$ and any $\lambda \in (0,1)$, the following identities hold:
\begin{enumerate}
\item\label{Hilbert-id-1} $\|\lambda x+(1-\lambda)y\|^2=\lambda\|x\|^2+(1-\lambda)\|y\|^2-\lambda(1-\lambda)\|x-y\|^2$;
\item\label{Hilbert-id-2} $\|x + y\|^2 = \|x\|^2 + \|y\|^2 + 2\langle x,y \rangle$ and 
$\|x - y\|^2 = \|x\|^2 + \|y\|^2 - 2\langle x,y \rangle$.
\end{enumerate}
\end{lemma}

We shall denote, for any $y,w\in C$,
$$\sigma(y,w):=\|w-Tw\|+\|y-Tw\|.$$

\begin{lemma}
Assume that $T$ is a pseudo-contraction. Then, for every $z,p\in C$, 
\begin{equation}
\|Tz-p\|^2 \leq \|z-p\|^2 + \|z-Tz\|^2 + 2\|p-Tp\|\sigma(z,p). \label{approz-triangle}
\end{equation}
\end{lemma}
\begin{proof}
Just follow the proof of \cite[Lemma 3.2.(i)]{IvaLeu15} (with $\kappa=1$). 
\end{proof}

The following equalities are immediate consequences of Lemma \ref{Hilbert-id}.\eqref{Hilbert-id-1}.

\begin{lemma}
For every $p\in C$, we have that:
\begin{eqnarray}
\|x_{n+1}-p\|^2 &=& \alpha_n\|Ty_n-p\|^2+(1-\alpha_n)\|x_n-p\|^2 -\alpha_n(1-\alpha_n)\|Ty_n-x_n\|^2
\label{id-0}\\
\|y_n-p\|^2 &=& \beta_n\|Tx_n-p\|^2+(1-\beta_n)\|x_n-p\|^2-\beta_n(1-\beta_n)\|Tx_n-x_n\|^2
\label{id-1}\\
\|y_n-Ty_n\|^2 &=& \beta_n\|Tx_n-Ty_n\|^2+(1-\beta_n)\|x_n-Ty_n\|^2-\beta_n(1-\beta_n)\|Tx_n-x_n\|^2
\label{id-2}
\end{eqnarray}
\end{lemma}

\begin{lemma}
Assume that $T$ is a pseudo-contraction and let $p\in C$.
\begin{enumerate}
\item  We have that:
\begin{eqnarray*} \|x_{n+1}-p\|^2 &\!\! \leq \!\!\ & \|x_n-p\|^2+\alpha_n\beta_n\|Tx_n-Ty_n\|^2-\alpha_n\beta_n(1-2\beta_n)\|Tx_n-x_n\|^2\\
&& -\alpha_n(\beta_n-\alpha_n)\|Ty_n-x_n\|^2+2\|p-Tp\|\big(\sigma(x_n,p)+\sigma(y_n,p)\big)
\end{eqnarray*}
\item Assume, furthermore,  that $T$ is $L$-Lipschitzian and that $(\alpha_n),(\beta_n)$ satisfy (A3). Then we have:
\begin{eqnarray}
\|x_{n+1}-p\|^2 &\leq & \|x_n-p\|^2-\alpha_n\beta_n(1-2\beta_n-L^2\beta_n^2)\|x_n-Tx_n\|^2 \notag\\
&& +2\|p-Tp\|\big(\sigma(x_n,p)+\sigma(y_n,p)\big). \label{useful-id-1}
\end{eqnarray}
\end{enumerate}
\end{lemma}
\begin{proof} 
The proof is a slightly modified version of the one from \cite{Ish74}. 
\begin{enumerate}
\item We get that:
{
\allowdisplaybreaks
\begin{eqnarray*}
\|x_{n+1}-p\|^2 &=& \alpha_n\|Ty_n-p\|^2+(1-\alpha_n)\|x_n-p\|^2  -\alpha_n(1-\alpha_n)\|Ty_n-x_n\|^2 \\
&& \text{by \eqref{id-0}}\\
& \leq & \alpha_n\big(\|y_n-p\|^2+\|y_n-Ty_n\|^2+2\|p-Tp\|\sigma(y_n,p)\big)+(1-\alpha_n)\|x_n-p\|^2 \\
&&  -\alpha_n(1-\alpha_n)\|Ty_n-x_n\|^2 \\
&& \text{by  \eqref{approz-triangle} with } z:=y_n\\
& = & \alpha_n\|y_n-p\|^2+(1-\alpha_n)\|x_n-p\|^2 -\alpha_n(1-\alpha_n)\|Ty_n-x_n\|^2\\
&& +\alpha_n\beta_n\|Tx_n-Ty_n\|^2+\alpha_n(1-\beta_n)\|x_n-Ty_n\|^2 \\
&& -\alpha_n\beta_n(1-\beta_n)\|Tx_n-x_n\|^2+ 2\alpha_n\|p-Tp\|\sigma(y_n,p)\\
&& \text{by \eqref{id-2}}\\
&=& \alpha_n\beta_n\|Tx_n-Ty_n\|^2+\alpha_n(\alpha_n-\beta_n)\|x_n-Ty_n\|^2+(1-\alpha_n)\|x_n-p\|^2\\
&& +\alpha_n\|y_n-p\|^2-\alpha_n\beta_n(1-\beta_n)\|Tx_n-x_n\|^2+2\alpha_n\|p-Tp\|\sigma(y_n,p)\\
&=& \alpha_n\beta_n\|Tx_n-Ty_n\|^2+\alpha_n(\alpha_n-\beta_n)\|x_n-Ty_n\|^2
+(1-\alpha_n)\|x_n-p\|^2\\
&& \alpha_n\big(\beta_n\|Tx_n-p\|^2+(1-\beta_n)\|x_n-p\|^2-\beta_n(1-\beta_n)\|Tx_n-x_n\|^2\big)\\
&&-\alpha_n\beta_n(1-\beta_n)\|Tx_n-x_n\|^2+2\alpha_n\|p-Tp\|\sigma(y_n,p)\\
&& \text{by \eqref{id-1}}\\
&=& \alpha_n\beta_n\|Tx_n-Ty_n\|^2+\alpha_n(\alpha_n-\beta_n)\|x_n-Ty_n\|^2+\|x_n-p\|^2\\
&& -2\alpha_n\beta_n(1-\beta_n)\|Tx_n-x_n\|^2+\alpha_n\beta_n(\|Tx_n-p\|^2-\|x_n-p\|^2)\\
&&+2\alpha_n\|p-Tp\|\sigma(y_n,p)\\
&\leq & \alpha_n\beta_n\|Tx_n-Ty_n\|^2+\alpha_n(\alpha_n-\beta_n)\|x_n-Ty_n\|^2+\|x_n-p\|^2\\
&&-2\alpha_n\beta_n(1-\beta_n)\|Tx_n-x_n\|^2+\alpha_n\beta_n\|Tx_n-x_n\|^2\\
&& +2\alpha_n\beta_n\|p-Tp\|\sigma(x_n,p)+2\alpha_n\|p-Tp\|\sigma(y_n,p)\\
&& \text{by  \eqref{approz-triangle} with } z:=x_n\\
&=& \|x_n-p\|^2 + \alpha_n\beta_n\|Tx_n-Ty_n\|^2-\alpha_n\beta_n(1-2\beta_n)\|Tx_n-x_n\|^2\\
&& -\alpha_n(\beta_n-\alpha_n)\|Ty_n-x_n\|^2+2\|p-Tp\|\big(\sigma(x_n,p)+\sigma(y_n,p)\big).
\end{eqnarray*}
\item If (A3) holds, then $\alpha_n(\beta_n-\alpha_n)\|Ty_n-x_n\|^2\geq 0$. It follows that:
\begin{eqnarray*} \|x_{n+1}-p\|^2 &\leq \ & \|x_n-p\|^2+\alpha_n\beta_n\|Tx_n-Ty_n\|^2-\alpha_n\beta_n(1-2\beta_n)\|Tx_n-x_n\|^2\\
&& +2\|p-Tp\|\big(\sigma(x_n,p)+\sigma(y_n,p)\big) \\
&\leq \ & \|x_n-p\|^2+L^2\alpha_n\beta_n\|x_n-y_n\|^2-\alpha_n\beta_n(1-2\beta_n)\|Tx_n-x_n\|^2\\
&& +2\|p-Tp\|\big(\sigma(x_n,p)+\sigma(y_n,p)\big) \\
&=& \|x_n-p\|^2+L^2\alpha_n\beta_n^3\|x_n-Tx_n\|^2-\alpha_n\beta_n(1-2\beta_n)\|Tx_n-x_n\|^2\\
&& +2\|p-Tp\|\big(\sigma(x_n,p)+\sigma(y_n,p)\big) \\
&& \text{by Remark~\ref{imm}}\\
&=& \|x_n-p\|^2+\alpha_n\beta_n(L^2\beta_n^2-1+2\beta_n)\|x_n-Tx_n\|^2\\
&& +2\|p-Tp\|\big(\sigma(x_n,p)+\sigma(y_n,p)\big) \\
&=& \|x_n-p\|^2-\alpha_n\beta_n(1-2\beta_n-L^2\beta_n^2)\|x_n-Tx_n\|^2\\
&& +2\|p-Tp\|\big(\sigma(x_n,p)+\sigma(y_n,p)\big).
\end{eqnarray*}
}
\end{enumerate}
\end{proof}

Let us recall some notions that are necessary for expressing our next results. 
Let $(a_n)_{n\in{\mathbb N}}$ be a sequence of nonnegative real numbers. If $(a_n)$ converges to $0$, 
then a {\em a rate of convergence} for $(a_n)$ is a mapping
$\alpha:{\mathbb N}\to{\mathbb N}$ such that:
\[\forall k\in{\mathbb N}\, \forall n\geq \alpha(k)\,\,\,  \left(a_n \leq \frac1{k+1}\right).\]

If the series $\sum\limits_{n=0}^\infty a_n$ diverges, then a
function $\theta:{\mathbb N}\to{\mathbb N}$ is called a {\em rate of divergence} 
of the series if for all $n \in {\mathbb N}$ we have that:
\[\sum_{i=0}^{\theta(n)}a_i \geq n.\]

A  {\em modulus of liminf} of $(a_n)$  is a mapping 
$\Delta:{\mathbb N}\times {\mathbb N}\to{\mathbb N}$  satisfying
\[\forall l\in{\mathbb N}\,\, \forall k\in{\mathbb N}\,\, \exists N \in [l,\Delta(l,k)] \,\,\,\,  \left( a_N  \leq \frac1{k+1}\right).\] 

One can easily see that $\displaystyle \liminf_{n\to\infty} a_n=0$ if and only if $(a_n)$ has a modulus of liminf.

In the situation where the nonnegative sequence is of the form $(\|x_n - Tx_n\|)$, we are 
often interested in a map $\Phi: {\mathbb N} \to {\mathbb N}$ such that :
$$\forall k\in{\mathbb N}\, \exists N \leq \Phi(k)\,\,\, \left(\|x_n - Tx_n\|  \leq \frac1{k+1}\right).$$
It is clear that such a map may be obtained from a modulus of liminf of $(\|x_n - Tx_n\|)$ 
by setting $l:=0$. Since its existence indicates that the elements of the sequence 
$(x_n)$ come arbitrarily close to being fixed points of the operator $T$,  
$\Phi$ is called an {\it approximate fixed point bound} for $(x_n)$ with respect to $T$.

\begin{lemma}\label{lem-betan}
Assume that ($\beta_n$) satisfies (A1) and that $\beta$ is a rate of convergence of $(\beta_n)$.  Set
\begin{equation}
 K:= \beta\left(\left\lceil 1+\sqrt{2L^2+4}\right\rceil\right). \label{def-nlg}
\end{equation} 
Then, for all $n\geq K$, $1-2\beta_n - L^2\beta_n^2 \geq \frac{1}{2}$.
\end{lemma}
\begin{proof}
Take $n\geq K$. Since $\beta$ is a rate of convergence for the nonnegative sequence $(\beta_n)$, 
whose limit is $0$, we have that $\beta_n \leq \frac{1}{1+\lceil 1+\sqrt{2L^2+4}\rceil}\leq 
\frac{1}{2+\sqrt{2L^2+4}} = \frac{-2+\sqrt{2L^2+4}}{2L^2}.$
It follows that  $\beta_n+\frac{1}{L^2}\leq\frac{\sqrt{2L^2+4}}{2L^2}$, so 
$\beta_n^2+\frac{2}{L^2}\beta_n + \frac{1}{L^4} \leq \frac{1}{2L^2} + \frac{1}{L^4}$ and
$L^2\beta_n^2+2\beta_n\leq \frac12$, hence the desired inequality. 
\end{proof}

Let us, for all $n\in{\mathbb N}$, denote:
\begin{equation}
z_n:=x_{n+K}.\label{def-zn}
\end{equation}
In particular, we have that $(z_n)$ is a subsequence of $(x_n)$.

\begin{lemma}\label{lem-useful-last}
Assume that $T$ is an $L$-Lipschitzian pseudo-contraction, $(\alpha_n)$, $(\beta_n)$ 
satisfy (A1) and (A3) and $\beta$ is a rate of convergence of $(\beta_n)$.
\begin{enumerate}
\item If $C$ is bounded and $b$ is an upper bound on the diameter of $C$, then for all $n\in{\mathbb N}$ and all $p\in C$, 
\begin{equation}
\|z_{n+1}-p\|^2 \leq  \|z_n-p\|^2-\frac12\alpha_n\beta_n\|z_n-Tz_n\|^2 +8b\|p-Tp\|.
\label{id-7}
\end{equation}
\item If $p$ is a fixed point of $T$, then for all $n\in{\mathbb N}$,
\begin{equation}
\|z_{n+1}-p\|^2 \leq  \|z_n-p\|^2-\frac12\alpha_n\beta_n\|z_n-Tz_n\|^2.
\label{id-8}
\end{equation}
\end{enumerate}
\end{lemma}
\begin{proof}
Apply Lemma \ref{lem-betan} and \eqref{useful-id-1}. For (i) use the fact that 
$2\|p-Tp\|\big(\sigma(x_n,p)+\sigma(y_n,p)\big)\leq 8b\|p-Tp\|$.
\end{proof}

\section{An effective modulus of liminf} \label{modulus-liminf}

In this section $C$ is a nonempty convex subset of a Hilbert space $H$, $T:C\to C$ 
is an $L$-Lipschitzian pseudo-contraction, $(\alpha_n),(\beta_n)$ are sequences in $[0,1]$ 
and $(x_n)$ is the Ishikawa iteration starting with $x\in C$.

The following result is the first step in Ishikawa's proof of Theorem \ref{main-thm-Ishikawa}.

\begin{proposition}\label{Ishikawa-liminf}
Assume that $T$ has fixed points and that $(\alpha_n)$, $(\beta_n)$ satisfy (A1)-(A3).  Then $\displaystyle  \liminf_{n\to\infty}\|x_n-Tx_n\|= 0$ for all $x\in C$.
\end{proposition}

The main result of this section is the following quantitative version of Proposition \ref{Ishikawa-liminf}, 
giving us an effective and uniform modulus of liminf for $(\|x_n-Tx_n\|)$.

\begin{theorem}\label{Ishikawa-liminf-quant}
Assume that $T$ has fixed points and that $(\alpha_n)$, $(\beta_n)$ satisfy (A1)-(A3). Let  
$\beta$ be a rate of convergence of $(\beta_n)$ and $\theta$ be a 
rate of divergence of $\sum\limits_{n=0}^{\infty}\alpha_n\beta_n$.

Let us define $\Delta_{b,\theta},\, \tilde{\Delta}_{b,L,\beta,\theta}:{\mathbb N}\times {\mathbb N}\to {\mathbb N}$ by 
\begin{eqnarray*}
\Delta_{b,\theta}(l,k):=\theta(l+M),  \quad  \tilde{\Delta}_{b,L,\beta,\theta}(l,k) =K+\Delta_{b,\theta}(l,k),
\end{eqnarray*}
with $\displaystyle  K:= \beta\left(\left\lceil 1+\sqrt{2L^2+4}\right\rceil\right)$,
$M:=2(b^2+1)(k+1)^2$ and $b\in {\mathbb N}$ is such that $b\geq \|x_K-p\|$ for some fixed point $p$ of $T$.

Then  for all $x\in C$,
\begin{enumerate}
\item  $\displaystyle  \liminf_{n\to\infty} \|z_n - Tz_n\|=0$ with modulus of liminf $\Delta_{b,\theta}$;
\item  $\displaystyle  \liminf_{n\to\infty} \|x_n - Tx_n\|=0$ with modulus of liminf $\tilde{\Delta}_{b,L,\beta,\theta}$.
\end{enumerate}
\end{theorem}
\begin{proof}
Let $x\in C$, $p\in Fix(T)$ and $b$ as in the hypothesis. We denote, for simplicity, $\Delta:=\Delta_{b,\theta}(l,k)$.
\begin{enumerate}
\item  We have to prove that 
\begin{equation}
\forall l\in{\mathbb N}\,\,\forall k\in{\mathbb N} \,\, \exists N\in[l,\Delta]\,\,\,\left(\|z_N - Tz_N\| \leq \frac1{k+1} \right).
\label{quant-modinf-eq}
\end{equation}

Remark first that, since $\theta$ is a rate of divergence for $\sum_{n=0}^{\infty}\alpha_n\beta_n$ and $\alpha_n,\beta_n$ 
are sequences in $[0,1]$, we have that $\theta(n)\geq n-1$ for all $n\in {\mathbb N}$. Then $\Delta\geq l+M-1\geq l$, 
as $M\geq 1$.

By \eqref{id-8}, we get that for all $n\in{\mathbb N}$, 
\begin{equation}
\|z_{n+1}-p\|^2 \leq \|z_n-p\|^2-\frac12\alpha_n\beta_n\|x_n-Tx_n\|^2. \label{liminf-1}
\end{equation}
As an immediate consequence, it follows that $\|z_{n+1}-p\|\leq \|z_n-p\|$ for all $n\in{\mathbb N}$. Thus,
$b\geq \|x_K-p\|=\|z_0-p\|\geq \|z_n-p\|$ for all $n\in{\mathbb N}$.

Assume by contradiction that \eqref{quant-modinf-eq} does not hold, hence $\|z_n-Tz_n\|> \frac1{k+1}$ 
for all $n\in [l,\Delta ]$. Adding \eqref{liminf-1} for $n:=l, \ldots, \Delta$, we get that 
\begin{eqnarray*}
   \|z_{\Delta+1}-p\|^2 &\leq &  \|z_l- p\|^2 - \frac{1}{2}\sum_{n=l}^{\Delta}
   \alpha_n\beta_n \|z_n - Tz_n\|^2 \leq b^2 - \frac1{2(k+1)^2}\sum_{n=l}^{\Delta}\alpha_n\beta_n. 
\end{eqnarray*}	
Remark now that 
\begin{eqnarray*}
   \sum_{n=l}^{\Delta}\alpha_n\beta_n = \sum_{n=0}^{\theta(l+M)}\alpha_n\beta_n - 
   \sum_{n=0}^{l-1}\alpha_n\beta_n
   \geq l+M - l = M.
\end{eqnarray*}
It follows that 
\[
   \|z_{\Delta+1}-p\|^2 \leq b^2 - \frac1{2(k+1)^2}M =-1.
\]
We have obtained a contradiction.
\item By (i), there exists $N\in[l,\Delta]$ such that 
\eqref{quant-modinf-eq} holds.  Let $\tilde{N}:=K+N$. Then 
$l\leq N\leq \tilde{N}\leq K+\Delta=K+\Delta_{b,\theta}(l,k)= \tilde{\Delta}_{b,L,\beta,\theta}(l,k) $ and $x_{\tilde{N}}=z_N$, so
$$\|x_{\tilde{N}} - Tx_{\tilde{N}}\|=\|z_N - Tz_N\| \leq \frac1{k+1}.$$
\end{enumerate}
\end{proof}

\begin{fact}
If $C$ is bounded, then, obviously, the above theorem holds with $b\in{\mathbb N}$ being an upper bound on the diameter of $C$.
\end{fact}

We get some immediate consequences. 

\begin{corollary}\label{afp-bound-xn-zn}
In the hypotheses of the above theorem, $\Delta'_{b,\theta}:{\mathbb N}\to{\mathbb N}$ is an approximate fixed point 
bound (with respect to $T$) for $(z_n)$ and
$\tilde{\Delta'}_{b,L,\beta,\theta}:{\mathbb N}\to{\mathbb N}$ is an approximate fixed point bound for $(x_n)$, where
\begin{eqnarray*}
\Delta'_{b,\theta}(k) &:=& \Delta_{b,\theta}(0,k)=\theta(M), \text{ and}\\
 \tilde{\Delta'}_{b,L,\beta,\theta}(k)  &:=& \tilde{\Delta}_{b,L,\beta,\theta}(0,k)=K+\theta(M). 
\end{eqnarray*}
\end{corollary}
\begin{proof}
As indicated before, we may just let $l:=0$ in the above theorem.
\end{proof}

In the case when $\alpha_n=\beta_n=\frac1{\sqrt{n+1}}$ we get a modulus of liminf of exponential growth.

\begin{corollary}\label{cor-exp}
In the hypotheses of the above theorem, assume further that $\alpha_n=\beta_n=\frac1{\sqrt{n+1}}$. Then, for all $x\in C$, 
$\displaystyle  \liminf_{n\to\infty} \|x_n - Tx_n\|=0$ with modulus of liminf $\Gamma_{b,L}$, given by:
\[\Gamma_{b,L}(l,k):= \left(\left\lceil 1+\sqrt{2L^2+4}\right\rceil + 1\right)^2 +4^{l+2(b^2+1)(k+1)^2}.\]
\end{corollary}
\begin{proof}
One can easily see that $\beta(k):=(k+1)^2$ is a rate of convergence for $\left(\beta_n=\frac1{\sqrt{n+1}}\right)$ 
and that $\theta(n) := 4^n$ is a rate of divergence for the sequence $\left(\alpha_n\beta_n = \frac{1}{n+1}\right)$. 
\end{proof}

\begin{corollary}
In the hypotheses of the above theorem, we have that for all $x \in C$, $\displaystyle  \liminf_{n\to\infty} \|x_n - x_{n+1}\|=0$ with modulus of liminf $\hat{\Delta}_{b,L,\beta,\theta}$, given by:
\[\hat{\Delta}_{b,L,\beta,\theta}(l,k) := \tilde{\Delta}_{b,L,\beta,\theta}(l, k'),\]
where $k':=\lceil(1+L)(1+k)\rceil$.
\end{corollary}
\begin{proof}
We know that there is an $N \in [l, \hat{\Delta}_{b,L,\beta,\theta}(l,k')]$ such that $\|x_N - Tx_N\| \leq \frac1{k'+1}$. 
Applying Lemma \ref{lem-1plusL}, we get that 
$$\|x_N - x_{N+1}\| \leq (1+L)\|x_N - Tx_N\| \leq \frac{1+L}{k'+1} \leq \frac1{k+1},$$ 
which was what we needed to show.
\end{proof}

An important class of pseudo-contractions are the $\kappa$-strict pseudo-contractions (where $0\leq \kappa < 1$), 
introduced also in \cite{BroPet67}. They are defined as mappings $T:C\to C$,
satisfying, for all $x,y\in C$,
\begin{equation} 
 \|Tx - Ty\|^2 \leq \|x-y\|^2 + \kappa \|x - Tx- \left(y -T y \right)\|^2. \label{def-k-strict-pseudo-contraction}
\end{equation}
It was proved in \cite[Proposition 2.1.(i)]{MarXu07} that any $\kappa$-strict pseudo-contraction is $L$-Lipschitzian with 
$L:=\frac{1+\kappa}{1-\kappa}$.
Furthermore, one can easily see that nonexpansive mappings coincide with $0$-strict pseudo-contractions. Thus, as 
a consequence of Theorem \ref{Ishikawa-liminf-quant} we get moduli of liminf for $(\|x_n-Tx_n\|)$ when $T$ belongs to 
these classes of mappings, too.  

\section{Uniform closedness and uniform generalized Fej\'er monotonicity}\label{Fejer-closed}

It was shown in \cite{KohLeuNic15} how one may derive the corresponding quantitative results
of a class of theorems stating the strong convergence of iterative algorithms. In the proofs 
of these theorems, compactness goes hand in hand with a property that the iterations typically 
exhibit (to some degree), called Fej\'er monotonicity, so the idea consists in exploiting this notion as much
as possible in order to replace the original arguments with purely
computational ones. It is this strategy that we shall use in the last section in order to obtain our main result. 
Firstly, however, we need to recall some essential notions from \cite{KohLeuNic15}.\\[1mm] 
Let $C$ be a nonempty subset of $H$ and $T:C\to C$ be a mapping with $Fix(T)\ne\emptyset$.

\begin{notation}
We denote $F:=Fix(T)$. 
\end{notation}

We may write $F:=\bigcap_{k \geq 0} AF_k$, where $AF_k$ is the set of all points $x \in C$ such that $\|x-Tx\| \leq \frac1{k+1}$.

The following uniform version of closedness was introduced in a more general context in \cite{KohLeuNic15}.

\begin{definition}
$F$ is called {\em uniformly closed with moduli}
$\delta_F,\omega_F:{\mathbb N}\to{\mathbb N}$  if for all $k\in{\mathbb N}$ and for all $p,q\in C$, 
\[
\|q-Tq\|\leq \frac1{\delta_F(k)+1} \text{~and~} \|p-q\|\le \frac1{\omega_F(k)+1} \quad \text{ imply } \quad 
\|p-Tp\|\leq \frac1{k+1}.
\]
\end{definition}
As pointed out in \cite[Lemma 7.1]{KohLeuNic15}, if $T$ is a uniformly continuous mapping, then  $F$ is uniformly closed
with moduli $\omega_F(k)=\max\{4k+3,\omega_T(4k+3)\}$ and 
$\delta_F(k)=2k+1$, where $\omega_T$ is a modulus of uniform continuity of $T$ -- that is, a mapping 
$\omega_T:{\mathbb N}\to{\mathbb N}$  such that  
\[
\|p-q\|\leq \frac1{\omega_T(k)+1} \quad \text{ implies }\quad  \|Tp-Tq\|\leq  \frac1{k+1} 
\]
for all $k\in{\mathbb N}$ and all $p,q\in C$.

\begin{proposition}\label{F-unif-closed-meta}
Assume that  $T$ is an $L$-Lipschitzian 
pseudo-contraction with $F\ne\emptyset$. Then $F$ is a uniformly closed subset of $C$ with moduli 
\[\omega_F(k)=\lceil L\rceil(4k+4) \quad \text{ and }\quad
\delta_F(k)=2k+1.\]
\end{proposition}
\begin{proof} Since $T$ is $L$-Lipschitzian, it follows immediately that $T$ is uniformly continuous with modulus
$\omega_T(k)=\lceil L\rceil(k+1)$. 
\end{proof}

Given two functions $G,H: {\mathbb R}_+\to {\mathbb R}_+$, a sequence $(u_n)$ in $C$ is said to be 
$(G,H)$-Fej\'er monotone w.r.t. $F$ if for all $n,m\in{\mathbb N}$ and all $p\in F$, 
\[H(\|u_{n+m}-p\|)\le G(\|u_n-p\|).\]
This is a natural generalizations of Fej\'er monotonicity, 
which is obtained by putting $G=H=id_{{\mathbb R}^+}$. 
As in  \cite{KohLeuNic15}, we suppose that the mappings $G,H$ satisfy the following 
properties: for all sequences $(a_n)$ in ${\mathbb R}_+$,
\begin{eqnarray*} (G) & \displaystyle \lim_{n\to\infty} a_n=0  \text{ implies } \displaystyle \lim_{n\to\infty} G(a_n)=0  \quad \text{and} \quad (H) & \displaystyle \lim_{n\to\infty} H(a_n)=0   
\mbox{ implies }  \displaystyle \lim_{n\to\infty} a_n=0.
\end{eqnarray*}
These properties allow us to obtain in the general setting some nice properties of Fej\'er monotone sequences,
needed for proving strong convergence.

Equivalent quantitative versions of $(G)$ and $(H)$ assert the existence of moduli $\alpha_G:{\mathbb N}\to{\mathbb N}$  and 
$\beta_H:{\mathbb N} \to{\mathbb N}$ such that for all $k\in{\mathbb N}$ and all $a\in{\mathbb R}_+$, 
\begin{eqnarray*}
 a\le \frac1{\alpha_G(k)+1} \text{ implies } G(a) \le \frac1{k+1}\qquad \text{and} \qquad
H(a)\le \frac1{\beta_H(k)+1} \text{ implies }  a \le  \frac1{k+1}.
\end{eqnarray*}
We say that $\alpha_G$ is a $G$-modulus and $\beta_H$ is an $H$-modulus.

The following uniform version of $(G,H)$-Fej\'er monotonicity was introduced in \cite{KohLeuNic15} 
and is another of the abovementioned notions needed to get our quantitative results.

\begin{definition}\label{dfn-fejer}
A sequence $(u_n)$ in $C$ is called {\em uniformly $(G,H)$-Fej\'er monotone} w.r.t. 
$F$ with modulus $\chi : {\mathbb N}^3 \to {\mathbb N}$ if for all $n,m,r\in{\mathbb N}$, for all $p\in C$ with 
$\|p-Tp\|\leq \frac{1}{\chi(n,m,r)+1}$ and for all $l \leq m$ we have that
\[
H(\|u_{n+l}-p\|))< G(\|u_n-p\|)+\frac{1}{r+1}.
\]
\end{definition}

\begin{proposition}\label{unif-Fej-zn}
Let $C\subseteq H$ be a bounded convex subset, $T:C\to C$ be  
an $L$-Lipschitzian pseudo-contraction  with $F\ne\emptyset$ and $b\in {\mathbb N}$ be an 
upper bound on the diameter of $C$. Assume that  $(\alpha_n)$, $(\beta_n)$ 
satisfy (A1) and (A3) and that $\beta$ is a rate of convergence of $(\beta_n)$. 
Then $(z_n)$ is uniformly $(G,H)$-Fej\' er monotone w.r.t. $F$ with modulus 
\[\chi_b(n,m,r)=8bm(r+1),\]
where $G(a)=H(a)=a^2$. We note that $\alpha_G(k)=\left\lceil \sqrt{k}\right\rceil$ is a $G$-modulus for $G$ 
and that $\beta_H(k)=(k+1)^2$ is a $H$-modulus for $H$.
\end{proposition}
\begin{proof}
Let $n,m,r\in{\mathbb N}, l\leq m$ and $p\in C$ be such that 
$\|p-Tp\|\leq \frac{1}{\chi(n,m,r)+1}=\frac{1}{8bm(r+1)+1}$.
As a consequence of \eqref{id-7}, we get that 
\begin{equation}
\|z_{n+1}-p\|^2 \leq  \|z_n-p\|^2 + 8b\|p-Tp\|.\label{lema-fejer-1}
\end{equation}
It follows that 
\begin{eqnarray*}
\|z_{n+l}-p\|^2 &\leq & \|z_n-p\|^2 + 8bl\|p-Tp\| \quad \text{(by induction from \eqref{lema-fejer-1})} \\
&\leq & \|z_n-p\|^2 + 8bm\|p-Tp\| \leq \|z_n-p\|^2 + \frac{8bm}{8bm(r+1)+1} \\
& < & \|z_n-p\|^2 +\frac{1}{r+1}.
\end{eqnarray*}
\end{proof}

\section{A rate of metastability}

In this section we give the main result of the paper, namely a finitary, quantitative version of 
Theorem \ref{main-thm-Ishikawa}. As we have already pointed out, we apply methods developed in \cite{KohLeuNic15} for obtaining
quantitative versions of generalizations of strong convergence results using  Fej\'er monotone 
sequences in totally bounded sets. \\[1mm] 
First, let us recall that a {\it modulus of total boundedness} for a nonempty subset $C\subseteq H$  
is a mapping $\gamma:{\mathbb N}\to{\mathbb N}$ such that for any $k\in{\mathbb N}$ and any sequence $(u_n)$ in $C$ we have that:
\[
\exists \,0\leq i<j\le \gamma(k)\,\,\left(\|u_i-u_j\|\le \frac{1}{k+1}\right).
\]
As pointed out in \cite{KohLeuNic15}, 
where two different moduli are considered, $C$ is totally bounded if and only if $C$ has a modulus of total boundedness.
This quantitative version of total boundedness was used in \cite{Ger08} to obtain, also using proof 
mining, quantitative results in topological dynamics. \\[1mm] 
For any function $f:{\mathbb N} \to {\mathbb N}$, define the function $f^M : {\mathbb N} \to {\mathbb N}$ by:
$$f^M(n):=\max_{0 \leq i \leq n} f(i).$$
Obviously, $f^M\geq f$ and $f$ is nondecreasing.\\[1mm] 
A {\it rate of metastability} for a sequence $(u_n)$ is a 
functional $\Sigma: {\mathbb N} \times {\mathbb N}^{\mathbb N} \to {\mathbb N}$ such that for any $k \in {\mathbb N}$ and any $g: {\mathbb N} \to {\mathbb N}$, the followings holds:

$$ \exists N \leq \Sigma(k,g) \,\forall i,j \in [N, N+g(N)]\,\,\, \left(\|x_i - x_j\| \leq \frac1{k+1}\right).$$

We now proceed to state our main result. Its proof can be found in the last subsection.

\begin{theorem}\label{main-thm-meta}
Let $H$ be a Hilbert space, $C\subseteq H$ a nonempty totally bounded convex subset, $T:C\to C$
an $L$-Lipschitzian pseudo-contraction with $F:=Fix(T)\ne\emptyset$, $(\alpha_n)$, $(\beta_n)$  sequences in $[0,1]$ satisfying  (A1)-(A3) and 
$(x_n)$ be the Ishikawa iteration starting with $x\in C$.
Assume, furthermore, that $\gamma$ is a modulus of total boundedness for $C$, $b\in{\mathbb N}$ is an upper bound on the diameter of $C$, 
$\beta$ is a 
rate of convergence of $(\beta_n)$ and $\theta$ is a rate of divergence of $\sum_{n=0}^{\infty}\alpha_n\beta_n$.

Let $\Sigma_{b,\theta,\gamma,\beta,L}$ and $\Omega_{b,\theta,\gamma,\beta,L}:{\mathbb N}\times{\mathbb N}^{\mathbb N}\to {\mathbb N}$ be defined as in Table \ref{tabel-1}. Then 
\begin{enumerate}
\item $\Sigma_{b,\theta,\gamma,\beta,L}$ is a rate of metastability for $(x_n)$.
\item There exists $N\le \Omega_{b,\theta,\gamma,\beta,L}(k,g)$ such that 
\[\forall i,j\in [N,N+g(N)]\
 \left( \|x_i-x_j\|\le \frac1{k+1} \text{~and~} \|x_i-Tx_i\|\le \frac1{k+1}  \right). \]
\end{enumerate}
\end{theorem} 

\begin{table}[ht!]
\begin{center}
\scalebox{0.95}{\begin{tabular}{ | l |}
\hline
\ \\
$\Sigma_{b,\theta,\gamma,\beta,L}(k,g):=K+\tilde{\Sigma}_{b,\theta,\gamma}(k,h)$, \\[2mm]
\quad $\tilde{\Sigma}_{b,\theta,\gamma}:{\mathbb N}\times{\mathbb N}^{\mathbb N}\to {\mathbb N}, \quad \tilde{\Sigma}_{b,\theta,\gamma}(k,g):=(\tilde{\Sigma}_0)_{b,\theta}(P,k,g)$, \\[2mm]
\quad $(\tilde{\Sigma}_0)_{b,\theta}:{\mathbb N}\times {\mathbb N}\times{\mathbb N}^{\mathbb N}\to {\mathbb N}$, \quad $\displaystyle  (\tilde{\Sigma}_0)_{b,\theta}(0,k,g):= 0$,\\[2mm]
\quad $(\tilde{\Sigma}_0)_{b,\theta}(n+1,k,g) := \theta^M\bigg(2(b^2+1)\big(8b(8k^2+16k+10)g^M\big((\tilde{\Sigma}_0)_{b,\theta}(n,k,g)\big)+1\big)^2\bigg)$,\\[3mm]
$\Omega_{b,\theta,\gamma,\beta,L}(k,g):=K+\tilde{\Omega}_{b,\theta,\gamma,L}(k,h)$,\\[2mm]
\quad $\tilde{\Omega}_{b,\theta,\gamma,L}:{\mathbb N}\times{\mathbb N}^{\mathbb N}\to {\mathbb N}, \quad\tilde{\Omega}_{b,\theta,\gamma,L}(k,g):=(\tilde{\Omega}_0)_{b,\theta,L}(P_0,k,g)$, \\[2mm]
\quad $(\tilde{\Omega}_0)_{b,\theta,L}:{\mathbb N}\times {\mathbb N}\times{\mathbb N}^{\mathbb N}\to {\mathbb N}$, \quad $\displaystyle  (\tilde{\Omega}_0)_{b,\theta,L}(0,k,g):= 0$,\\[2mm]
\quad $(\tilde{\Omega}_0)_{b,\theta,L}(n\!+\! 1,k,g) := \theta^M\!\bigg(2(b^2\!+\! 1)\big(\max\{2k\!+\! 1,8b(8k_0^2\!+\! 16k_0\!+\! 10)g^M((\tilde{\Omega}_0)_{b,\theta,L}(n,k,g))\}\!+\! 1\big)^2\!\bigg)$,\\[3mm]
$K:= \beta\left(\left\lceil 1+\sqrt{2L^2+4}\right\rceil\right)$, \qquad $h(n):=g(K+n)$,\\[2mm]
 $\displaystyle  P :=\gamma\left(\left\lceil \sqrt{8k^2+16k+9}\right\rceil \right)$, \qquad $\displaystyle  k_0 := \left\lceil\frac{\lceil L \rceil (4k+4) -1}{2}\right\rceil$, 
 \qquad $\displaystyle  P_0 :=\gamma\left(\left\lceil \sqrt{8k_0^2+16k_0+9}\right\rceil \right)$.\\[2mm]
\ \\
\hline
\end{tabular}}
\caption{Functionals and constants.}\label{tabel-1}
\end{center}
\end{table}

Theorem~\ref{main-thm-meta}.(i) gives us a highly uniform rate of metastability 
$\Sigma_{b,\theta,\gamma,\beta,L}$, which depends only on the Lipschitz constant $L$, an upper bound $b$ on the diameter of $C$ and a modulus of total
boundedness $\gamma$ for $C$, and the rates $\beta,\theta$ associated to the sequences $(\alpha_n),(\beta_n)$. 
As an immediate consequence, we get the Cauchyness of $(x_n)$ for totally bounded  convex $C$. 
Using \cite[Remark 5.5]{KohLeuNic15}, we may see that Theorem~\ref{main-thm-meta}.(ii) is 
indeed the true finitization of Ishikawa's original statement, i.e. it implies back not only 
the convergence of the iterative sequence, but also the fact that its limit point is a fixed point of $T$.

\begin{corollary}
In the hypotheses of the above theorem, assume further that $\alpha_n=\beta_n=\frac1{\sqrt{n+1}}$.
Then there exists $N\le \Omega'_{b,\gamma,L}(k,g)$ such that 
\[\forall i,j\in [N,N+g(N)]\
 \left( \|x_i-x_j\|\le \frac1{k+1} \text{~and~} \|x_i-Tx_i\|\le \frac1{k+1}  \right), \]
\end{corollary}
where $\Omega'_{b,\gamma,L}(k,g):=K_0+(\Omega'_0)_{b,L}(P_0,k,h)$, with 
$K_0  := \left(\left\lceil 1+\sqrt{2L^2+4}\right\rceil + 1\right)^2$,
\begin{eqnarray*} 
(\Omega'_0)_{b,L}(0,k,g) &:=& 0,\\
(\Omega'_0)_{b,L}(n+1,k,g) &:=& 4^{2(b^2+1)\big(\max\{2k+1,8b(8k_0^2+16k_0+10)g^M((\Omega'_0)_{b,L}(n,k,g))\}+1\big)^2}.
\end{eqnarray*} 
and $h,P_0,k_0$ as in Table \ref{tabel-1}.
\begin{proof}
Use the moduli from Corollary~\ref{cor-exp}.
\end{proof}

\subsection{Proof of Theorem \ref{main-thm-meta}}
\begin{enumerate}
\item 
{\bf Claim:} $\tilde{\Sigma}_{b,\theta,\gamma}$ is a rate of metastability for $(z_n)$.\\[1mm]
{\bf Proof of claim:}  
By Proposition \ref{unif-Fej-zn}, $(z_n)$ is uniformly $(G,H)$-Fej\' er monotone  w.r.t. $F$ with modulus  
\[\chi_b(n,m,r)=8bm(r+1),\]
where $G(a)=H(a)=a^2$ with moduli 
$$\alpha_G(k)=\left\lceil \sqrt{k}\right\rceil \quad \text{and} \quad \beta_H(k)=(k+1)^2.$$

Define $\Phi:{\mathbb N}\to {\mathbb N}$ by 
\begin{equation}
\Phi(k):= \theta^M(2(b^2+1)(k+1)^2)
\end{equation}
Then $\Phi$ is nondecreasing and $\Phi$ is an an approximate fixed point bound for $(z_n)$ by Corollary \ref{afp-bound-xn-zn}
and the fact that $\Phi(k)\geq \theta(2(b^2+1)(k+1)^2)$ for all $k$. 

We may now apply \cite[Theorem 5.1]{KohLeuNic15}  for $F$ and $(z_n)$. Using the notations from  
\cite[Theorem 5.1]{KohLeuNic15}, we get in our setting that
\begin{eqnarray*}
\chi_g(n,k) = 8(k+1)bg(n), \quad  \chi^M_g(n,k) = 8(k+1)bg^M(n), \quad P =\gamma\left(\left\lceil \sqrt{8k^2+16k+9}\right\rceil \right)
\end{eqnarray*}
and
\begin{eqnarray*}
\Psi_0(0,k,g,\Phi,\chi,\beta_H)\!\!\!\!\!&=&\!\!\!\!0 \\
\Psi_0(n\!+\!1,k,g,\Phi,\chi,\beta_H)\!\!\!\!\!&=&\!\!\!\!\theta^M\!\bigg(2(b^2+1)\big(8b(8k^2+16k+10)g^M\big(\Psi_0(n,k,g,\Phi,\chi,\beta_H)\big)\!+1\!\big)^2\!\bigg).\\
\end{eqnarray*}
By induction, we have that $\Psi_0(n,k,g,\Phi,\chi,\beta_H) = (\tilde{\Sigma}_0)_{b,\theta}(n,k,g)$. It follows that 
$$\Psi(k,g,\Phi,\chi,\alpha_G,\beta_H,\gamma) = \Psi_0(P,k,g,\Phi,\chi,\beta_H) = (\tilde{\Sigma}_0)_{b,\theta}(P,k,g) = \tilde{\Sigma}_{b,\theta,\gamma}(k,g).$$

Thus, the claim is proved. $\blacksquare$\\

Let $k\in{\mathbb N}$ and $g:{\mathbb N}\to{\mathbb N}$ be arbitrary. Applying the claim, we get $N\leq \tilde{\Sigma}_{b,\theta,\gamma}(k,h)$ such that for 
all $i,j\in [N,N+h(N)]=[N,N+g(K+N)]$,
\[\|z_i-z_j\|\leq \frac1{k+1}.\]
Define  $\tilde{N}:=K+N$. Then $\tilde{N}\leq K+\tilde{\Sigma}_{b,\theta,\gamma}(k,h_g)=\Sigma_{b,\theta,\gamma,\beta,L}(k,g)$ and $x_{\tilde{N}}=z_N$. Let 
$i,j\in[\tilde{N}, \tilde{N}+g(\tilde{N})]=[K+N, K+N+g(K+N)]$ and take $i_0:=i-K, j_0:=j-K$. 
Then $i_0,j_0\in [N, N+g(K+N)]$ and so:
\begin{eqnarray*}
\|x_i-x_j\|=\|x_{i_0+K} - x_{j_0+K}\|=\|z_{i_0}-z_{j_0}\|\leq \frac1{k+1}.
\end{eqnarray*}
\item We apply now \cite[Theorem 5.3]{KohLeuNic15} for $F$ and $(z_n)$. Using the notations from  
\cite[Theorem 5.3]{KohLeuNic15} and using Proposition~\ref{F-unif-closed-meta} we get in our setting that
\begin{eqnarray*}
k_0 =\left\lceil\frac{\lceil L \rceil (4k+4) -1}{2}\right\rceil, \quad \chi_{k, \delta_F}(n,m,r) = 
\max\{2k+1, \chi(n,m,r)\} \\
\end{eqnarray*}
It follows that $\Psi_0(n,k_0,g,\Phi,\chi_{k, \delta_F},\beta_H) = (\tilde{\Omega}_0)_{b,\theta,L}(n,k,g)$ and 
$$\Psi(k_0,g,\Phi,\chi_{k, \delta_F},\alpha_G,\beta_H,\gamma) \!= \!
\Psi_0(P_0,k_0,g,\Phi,\chi_{k, \delta_F},\beta_H)\! =\! (\tilde{\Omega}_0)_{b,\theta,L}(P_0,k,g) \!=\! \tilde{\Omega}_{b,\theta,\gamma,L}(k,g).$$

Thus, we have obtained that for all $k\in{\mathbb N}$ and $g:{\mathbb N}\to{\mathbb N}$, there exists $N\le \tilde{\Omega}_{b,\theta,\gamma,L}(k,g)$ such that 
\[\forall i,j\in [N,N+g(N)]\
 \left( \|z_i-z_j\|\le \frac1{k+1} \text{~and~} \|z_i-Tz_i\|\le \frac1{k+1}  \right). \]
As in (i), one gets immediately that (ii) holds. 
\end{enumerate}

\mbox{}

\noindent
{\bf Acknowledgements:} \\[1mm] 
Lauren\c tiu Leu\c stean and Andrei Sipo\c s were supported by a grant of the Romanian 
National Authority for Scientific Research, CNCS - UEFISCDI, project 
number PN-II-ID-PCE-2011-3-0383.


\begin{thebibliography}{}



\bibitem{BroPet66} 
F.E. Browder, W.V. Petryshyn, The solution by iteration of nonlinear functional equations in Banach spaces, 
Bull. Amer. Math. Soc. 72 (1966), 571--575. 

\bibitem{BroPet67}
F.E. Browder, W.V. Petryshyn, Construction of fixed points of nonlinear mappings in Hilbert
spaces, J. Math. Anal. Appl. 20 (1967), 197-228.

\bibitem{Bru74}
R.E. Bruck Jr., A strongly convergent iterative solution of $0 \in U(x)$ for a maximal monotone operator $U$ in Hilbert space,
J. Math. Anal. Appl. 48:1 (1974), 114-126.

\bibitem{ChiMut01}
C.E. Chidume, S.A. Mutangadura, An example on the Mann iteration method for Lipschitz pseudo-contractions,
Proc. Amer. Math. Soc. 129 (2001), 2359-2363.

\bibitem{Ger08} 
P. Gerhardy, Proof mining in topological dynamics, Notre Dame J. Form. Log. 49 (2008), 431--446.


\bibitem{Ish74}
S. Ishikawa, Fixed points by a new iteration method, Proc. Amer. Math. Soc. 44 (1974), 147-150.

\bibitem{IvaLeu15}
D. Ivan, L. Leu\c stean, A rate of asymptotic regularity for the Mann iteration of $\kappa$-strict pseudo-contractions,
Numer. Funct. Anal. Optimiz. 36 (2015), 792-798.


\bibitem{Koh08-book}
U. Kohlenbach, Applied proof theory: Proof interpretations and their use in mathematics, 
Springer, Berlin/Heidelberg, 2008.

\bibitem{Koh16}
U. Kohlenbach, Recent progress in proof mining in nonlinear analysis, preprint, 2016.

\bibitem{KohLeuNic15} 
U. Kohlenbach, L. Leu\c{s}tean, A. Nicolae, Quantitative results on Fej\'er monotone sequences, 
arXiv:1412.5563 [math.LO], 2015. 

\bibitem{MarXu07} 
G. Marino, H.-K. Xu, Weak and strong convergence theorems for strict pseudo-contractions in Hilbert spaces, 
J. Math. Anal. Appl. 329 (2007), 336-346.

\bibitem{Leu10}
L. Leu\c{s}tean, Nonexpansive iterations in uniformly convex $W$-hyperbolic spaces, in: A. Leizarowitz, 
B. S. Mordukhovich, I. Shafrir, A. Zaslavski (Eds.), Nonlinear Analysis and Optimization I: Nonlinear 
Analysis, Cont. Math. 513, Amer. Math. Soc., Providence, RI, 2010, pp. 193-209.

\bibitem{Leu14}
L. Leu\c stean, An application of proof mining to nonlinear iterations, 
Ann. Pure Appl. Logic 165 (2014),  1484-1500.

\bibitem{Tao07} 
T. Tao, Soft analysis, hard analysis, and the finite convergence principle, Essay posted May 23, 2007, appeared in: 
T. Tao, Structure and Randomness: Pages from Year One of a Mathematical Blog,  Amer. Math. Soc., Providence, RI, 2008.

\bibitem{Tao08} 
T. Tao, Norm convergence of multiple ergodic averages for commuting transformations, Ergodic Theory Dynam. Systems 28 (2008), 657--688. 

\end{thebibliography}
\end{document}